\pdfoutput=1
\documentclass{amsart}
\usepackage{ucs}
\usepackage[utf8x]{inputenc}
\usepackage[icelandic, english]{babel}
\usepackage{t1enc}
\usepackage{graphicx}
\usepackage[intoc]{nomencl}
\usepackage{enumerate,color}
\usepackage{url}
\usepackage[pdfborder={0 0 0}]{hyperref}
\usepackage{appendix}
\usepackage{eso-pic}
\usepackage{amsmath}
\usepackage{amssymb}
\usepackage[square,sort,comma,numbers]{natbib}
\usepackage[sf,normalsize]{subfigure}
\usepackage[format=plain,labelformat=simple,labelsep=colon]{caption}
\usepackage{placeins}
\usepackage{tabularx}

\graphicspath{{figs/}}

\newcommand{\C}{{\mathbb  C}}
\newcommand{\N}{{\mathbb  N}}

\newcommand{\R}{{\mathbb  R}}

\newcommand{\Z}{{\mathbb  Z}}

\newcommand{\Det}{{\operatorname{Det}}}
\usepackage{amsthm}
\usepackage{mathrsfs}
\newtheorem{theorem}{Theorem}[section]

\newtheorem{lemma}[theorem]{Lemma}
\newtheorem{proposition}[theorem]{Proposition}

\newtheorem{definition}[theorem]{Definition}
\newtheorem*{remarks*}{Remarks}
\newtheorem*{remark*}{Remark}

\title{Rapid polynomial approximation on Stein manifolds}

\author{Au\dh unn Sk\'uta Sn\ae bjarnarson}

\begin{document}
\maketitle

\section*{}\small
\textbf{Abstract.} In this paper we generalize to a certain class of Stein manifolds the Bernstein-Walsh-Siciak theorem which describes the equivalence between possible holomorphic continuation of a function $f$ defined on a compact set $K$ in $\C^N$ to the rapidity of the best uniform approximation of $f$ on $K$ by polynomials. We also generalize Winiarski's theorem which relates the growth rate of an entire function $f$ on $\mathbb{C}^N$ to its best uniform approximation by polynomials on a compact set.

\normalsize

\section{Introduction}

The famous Runge-Oka-Weil theorem can be phrased in the following way:
\begin{center}
	If $K\subset \C^N$ is compact and polynomially convex, then\\ $\lim_{n\to\infty}d_K(f,\mathcal{P}_n)=0$ for every $f\in \mathcal{O}(K)$.
\end{center}
Here $\mathcal{P}_n$ is the set of polynomials in $\C^N$ of degree less than or equal to $n$ and $d_K(f,\mathcal{P}_n):=\inf\{\|f-p\|_K:\; p\in \mathcal{P}_n \}$ is the best uniform approximation of $f$ on $K$ by polynomials in $\mathcal{P}_n$. In \cite{Sic:1962} Siciak proves a precise quantitative version of the Oka-Weil-Runge theorem. More specifically, provided certain regularity conditions on $K$, he proves that 
\begin{align}
	\limsup_{n\to\infty} \left(d_K(f,\mathcal{P}_n)\right)^{1/n}\leq L^{-1}\label{abcdefg}
\end{align} 
if and only if $f$ extends as a holomorphic function to $\{z\in \C^N;\; V_K(z)<\log(L) \}$. Here $V_K$ denotes the Siciak-Zakharyuta extremal function defined as the supremum of all entire plurisubharmonic functions $u$ of minimal growth with $u|_K\leq 0$ (see for example \cite{Sic:1981} or \cite{Kli:1991}). Taking Siciak's theorem into consideration, we say that a function $f$ admits to rapid approximation by polynomials on $K$ if (\ref{abcdefg}) is true for some $L>1$. Goncar \cite{Gon:1972,Gon:1974,Gon:1975} and Cirka \cite{Cir:1976,Cir:1976:2} have proven theorems in a similar spirit, regarding rapid approximation by rational functions on $\C^N$ (defined in an analogous way).

Following Stoll \cite{Stoll:1977} we say that a manifold $X$ of complex dimension $N$ is {\it $S$-parabolic} if it possesses a plurisubharmonic (psh) exhaustion function $\tau$ which is maximal outside a compact subset $S$ of $X$. Such an exhaustion function is called a {\it special} exhaustion function. The function $\tau$ being maximal outside $S$ is equivalent to $(i\partial \bar{\partial}\tau)^N=0$ on $X\setminus S$. We say that an entire function $f$ on $X$ is a $\tau$-polynomial if there are constants $t,C\geq 0$ such that
\begin{align*}
	\log|f(z)|\leq t\tau^+(z)+C,\qquad z\in X.
\end{align*}
In their work \cite{Ayt:2015,Ayt:2011,Ayt:2014}, Aytuna and Sadullaev consider the Fréchet-space $\mathcal{O}(X)$ of holomorphic functions on $X$. They construct an example of an $S$-parabolic manifold where the $\tau$-polynomials are not dense in $\mathcal{O}(X)$. Zeriahi \cite{Zer:1991,Zer:1996} introduces analogues of classical pluripotential theory to $S$-parabolic Stein manifolds. He generalizes the theorem of Siciak and a theorem of Winiarski \cite{Win:1970} to algebraic varieties.

In this paper we consider a Stein manifold $X$ with a psh exhaustion function $\psi$ such that the $(1,1)$-form $\frac{i}{2}\partial \bar{\partial}e^\psi$ satisfies certain curvature properties (which we discuss in detail in Section 3). When the curvature properties in question are satisfied, we prove a theorem on rapid approximation by $\psi$-polynomials on compact subsets of $X$. We also prove a generalization of Winiarski's theorem to such manifolds. Our main results are stated in Section 3 and proven in Section 5. In Section 4 we look at a few examples of functions $\psi$ satisfying the aforementioned curvature properties.

\textbf{Acknowledgments.} The author would like to thank his doctoral advisor, Ragnar Sigurðsson, for reviewing this paper and giving helpful comments. This project was funded by The Doctoral Grants of The University of Iceland Research Fund and by The Icelandic Center for Research (Rannís) grant no. 152572-052.

\section{Preliminaries}

Let $X$ be a Stein manifold and $\psi:X\to \R$ be a plurisubharmonic exhaustion function. The notion of a $\psi$-polynomial on $X$ was introduced by Zeriahi in \cite{Zer:1991}.
\begin{definition}
	We say that a function $f\in \mathcal{O}(X)$ is a $\psi$-polynomial if there exist constants $t$ and $C$ such that
	\begin{align}
		\log|f(z)|\leq t\psi^+(z)+C,\qquad z\in X,\label{poly}
	\end{align} 
	where $\psi^+(z):=\max\{0,\psi(z) \}$ is the positive part of $\psi$. We denote by $\mathcal{P}^\psi$ the space of $\psi$-polynomials on $X$ and for a fixed $t>0$ we denote by $\mathcal{P}^\psi_{t}$ the set of $\psi$-polynomials on $X$ satisfying inequality (\ref{poly}) for some constant $C$. If $f$ is a $\psi$-polynomial on $X$ then the $\psi$-degree of $f$ is 
	\begin{align*}
		\deg_\psi(f):=\inf\{t>0:\; f\in\mathcal{P}^\psi_{t} \}.
	\end{align*}
\end{definition}

Note that if $X=\C^N$ and $\psi(z)=\log\|z\|$ then the notion of a $\psi$-polynomial coincides with the classical notion of a polynomial. The polynomial spaces $\mathcal{P}^\tau_{t}$ are of particular interest when the function $\tau$ is a special exhaustion because then they are of finite dimension. More specifically we have:
\begin{proposition}[\cite{Zer:1991}, Th\'eor\`eme 4.8]\label{dimensional}
	If $\tau$ is a special exhaustion function on the $N$-dimensional manifold $X$, there exists a constant $M$ such that
	\begin{align*}
		\dim \mathcal{P}^\tau_{n}\leq{{N +nM}\choose{N }},\qquad n\in \N.
	\end{align*}
\end{proposition}
In \cite{Zer:1991,Zer:1996} Zeriahi considers the case when $X$ is an affine algebraic variety and proves theorems similar to the theorems of Oka-Weil and Siciak. In \cite{Ayt:2011,Ayt:2014,Ayt:2015} Aytuna and Sadullaev consider the polynomial space $\mathcal{P}^\tau$ when $\tau$ is a special exhaustion function. They construct an example where the polynomial space $\mathcal{P}^\tau$ consist only of the constant functions, and another one where $\mathcal{P}^\tau$ is not trivial, but still not dense in the Fréchet-space of holomorphic functions $\mathcal{O}(X)$. As a corollary to our main results of this paper we find sufficient conditions for $\mathcal{P}^\tau$ to be dense in $\mathcal{O}(X)$.

As an analogue to the classical Lelong class $\mathcal{L}$ on $\C^N$ we define the $\psi$-Lelong class on $X$ to be the set
\begin{align*}
	\mathcal{L}_{\psi}:=\{u\in \operatorname{PSH}(X);\; \exists C\geq 0\;\text{ such that }\; u\leq \psi^++C\;\text{ on }\; X \}
\end{align*}
where $\psi$ is any psh exhaustion on $X$ and we define
\begin{align*}
	\mathcal{L}^+_{\psi}:=\{u\in \mathcal{L}_{\psi};\; \exists C\geq 0\;\text{ such that }\; \psi\leq u^++C\;\text{ on }\; X \}.
\end{align*}
If $\psi$ is a special exhaustion function then $\mathcal{L}_{\psi}$ is an abstract Lelong class in the sense of \cite{Zer:2000}. This means that for any compact non-pluripolar set $K\subset X$ the extremal function
\begin{align*}
	V_{K,\psi}(z):=\sup\{v(z);\; v\in \mathcal{L}_{\psi},\;\; v|_K\leq 0 \},\qquad z\in X
\end{align*} 
is well defined, i.e.\ we have $V_{K,\psi}<\infty$. Indeed, the upper semi-continuous regularization $V^*_{K,\psi}$ is a member of $\mathcal{L}_{\psi}$. We also define the function
\begin{align*}
	\Phi_{K,\psi}(z)=\sup\{|f(z)|^{1/t}; f\in \mathcal{P}^\psi_{t},\; \|f\|_K\leq 1,\; t>0 \}.
\end{align*}
In the case when $X=\C^N$ and $\psi(z)=\log\|z\|$ the function $\Phi_K:=\Phi_{K,\psi}$ was originally introduced by Siciak \cite{Sic:1962} in order to extend classical results of approximation and interpolation to holomorphic functions of several complex variables. Later, Zakharyuta \cite{Zak:1976} defined the extremal function $V_K:=V_{K,\psi}$ with $X=\C^N$ and $\psi=\log\|z\|$. It is well known that $\log\Phi_{K}=V_{K}$ for every compact $K\subset \C^N$ (see for example \cite[Theorem 5.1.7]{Kli:1991}), but on a more general manifold such an equality might not be true, even if we assume $\psi$ to be a special exhaustion function. Indeed, as mentioned before, there exists an example of a special exhaustion $\psi$ on a manifold $X$ such that $\mathcal{P}^\psi$ consists only of the constants \cite{Ayt:2015}, in which case we have $\log\Phi_{K,\psi}\equiv 0$. In general we have $\log\Phi_{K,\psi}\leq V_{K,\psi}$.

\section{Results}
In this section we present the main results of this paper. All results are proven in Section 5. First we must introduce some notation. Recall that if $\omega$ is a Kähler-form on $X$ with coefficients $\omega_{j,\overline{k}}$ with respect to a given coordinate system then the Ricci curvature of $\omega$ is given as follows
\begin{align*}
	\operatorname{Ricci}(\omega)=-i\partial \bar{\partial}\log(\Det(\omega_{j,\overline{k}})).
\end{align*}
For any $z\in X$, $r>0$ we denote by $B(z,r,\omega)$ the geodesic ball with center $z$ and radius $r$ with respect to the metric $\omega$.

\begin{definition}\label{skil}
	Let $\psi$ be a psh exhaustion function on the $N$-dimensional manifold $X$ and assume that the $(1,1)$-form $\frac{i}{2}\partial\bar{\partial}e^\psi$ is smooth and strictly positive outside a compact subset $S$ of $X$.
	\begin{enumerate}[(i)]
		\item Let $\theta\in \operatorname{PSH}(X)$. We say that $\theta$ is a Ricci compensator for $\psi$ if it is continuous, strictly psh in a neighborhood of $S$,
		\begin{align*}
			|\theta|\leq A\psi^++B\qquad \text{on }\; X
		\end{align*}
		for some constants $A,B>0$ and 
		\begin{align*}
			\tfrac{i}{2}\partial\bar{\partial}\theta+\operatorname{Ricci}\left(\tfrac{i}{2}\partial\bar{\partial}e^\psi\right)\geq 0,\qquad \text{on }\; X\setminus S.
		\end{align*}
		If there exists a Ricci compensator for $\psi$, then we say that $\psi$ is Ricci compensable.
		\medskip
		\item We say that $\psi$ induces an integral estimate for holomorphic functions if for every $\delta>0$ there are constants $A,B$ such that for every $z\in X\setminus S$ and every function $F\in \mathcal{O}(\overline{B}(z,\delta,\frac{i}{2}\partial\bar{\partial}e^\psi))$
		\begin{align*}
			|F(z)|^2\leq e^{A\psi^+(z)+B}\int_{B(z,\delta,\tfrac{i}{2}\partial \bar{\partial}e^\psi)\setminus S}|F|^2\left(\tfrac{i}{2}\partial\bar{\partial}e^\psi\right)^N.
		\end{align*}
	\end{enumerate}
\end{definition}
Our first main result is the following.
\begin{theorem}\label{adal0}
	Let $\psi$ be a psh exhaustion function on $X$ which is Ricci compensable and induces an integral estimate for holomorphic functions. Let $K\subset X$ be compact and $\varphi\in \mathcal{L}_\psi^+$ be a continuous function satisfying $\varphi|_K\leq 0$. Then for every $L\in ]1,\infty[$ and every function $f$ holomorphic on $K_L:=\{z\in X;\; \varphi(z)<\log(L) \}$ we have
	\begin{align*}
		\limsup_{t\to\infty}\left(d_K(f,\mathcal{P}^\psi_{t})\right)^{1/t}\leq L^{-1}.
	\end{align*}
\end{theorem}
Notice that if the function $\psi$ is a member of $\mathcal{L}_\tau$ for some special exhaustion function $\tau$ then the polynomial spaces $\mathcal{P}^\psi_{t}$ have finite dimension (this follows directly from Proposition \ref{dimensional}). This means that Theorem \ref{adal0} is, in some sense, the strongest in this case. We should note though, that for a given special exhaustion function $\tau$ it is not always possible to find $\psi\in \mathcal{L}_\tau$ satisfying the properties of Definition \ref{skil} (for instance if $\mathcal{P}^\tau$ consists only of constants). We then need larger polynomial spaces if we want to apply Theorem 3.2.

If $\psi\in \mathcal{L}_\tau$ for a special exhaustion function $\tau$, then the extremal function $V_{K,\psi}$ is well defined. If $V_{K,\psi}$ happens to be continuous as well, then we can take $\varphi$ to be equal to $V_{K,\psi}$ in Theorem \ref{adal0}. In this case the converse of Theorem \ref{adal0} is true as well:
\begin{proposition}\label{prop33}
	If $\psi\in \mathcal{L}_\tau$ for some special exhaustion function $\tau$ and $f:K\to \C$ is any function s.t.\ 
	\begin{align}
		\limsup_{t\to\infty}\left(d_K(f,\mathcal{P}^\psi_{t})\right)^{1/t}\leq L^{-1}\label{laks}
	\end{align}
	for some $L>1$, then $f$ is the restriction to $K$ of a function holomorphic on the set $K_L=\{z\in X;\; V_{K,\psi}(z)<\log(L) \}$.
\end{proposition}

Observe that if inequality (\ref{laks}) holds for every $L>0$, then $f$ is the restriction to $K$ of an entire function, also denoted by $f$. If $f$ is of finite order $\varrho$ and of finite type $\sigma$ with respect to $\varrho$, then we have a more precise estimate of $d_K(f,\mathcal{P}^\psi_t)^{1/t}$. More precisely, we have a generalization of a theorem of Winiarski \cite{Win:1970} for the special case $X=\C^N$ and $\psi=\log\|z\|$.

\begin{theorem}\label{adal00}
	Assume $\psi$ is Ricci compensable and induces an integral estimate for holomorphic function. Further assume that for every $r>0$ there exist constants $A,B$ such that
	\begin{align}\label{volgrowth}
		\int_{\{\psi(z)<\log(L) \}}\big(\tfrac{i}{2}\partial\bar{\partial} e^{\psi}\big)^N\leq e^{AL^r+B},\qquad L>1.
	\end{align}
	Let $K\subset X$ be compact and $\varphi\in \mathcal{L}_\psi^+$ be a continuous function on $X$ satisfying $\varphi|_K\leq 0$. Then for any entire function $f$ on $X$ satisfying the growth estimates
	\begin{align}
		\label{jafna1}
		\limsup_{r\to\infty}\frac{\log^+\log\|f\|_{\{\varphi\leq \log(r)\}}}{\log(r)}\leq\varrho\;\;\; \text{and}\;\;\;
		\limsup_{r\to\infty}\frac{\log\|f\|_{\{\varphi\leq \log(r)\}}}{r^{\varrho}}\leq\sigma,
	\end{align}
	for some $\varrho>0$, $\sigma\geq0$, we have
	\begin{align}
		\label{jafna3}
		\limsup_{t\to\infty}t^{1/\varrho}(d_K(f,\mathcal{P}^\psi_{t}))^{1/t}\leq (e\sigma\varrho)^{1/\varrho}.
	\end{align}
	If $\psi\in \mathcal{L}_{\tau}$ for some special exhaustion function $\tau$ and we take $\varphi:=V_{K,\psi}$ then the converse holds as well, i.e.\ if $f$ is a function on $K$ and inequality (\ref{jafna3}) holds, then $f$ extends to an entire function on $X$ and inequalities $(\ref{jafna1})$ are true with $\varphi$ replaced by $V_{K,\psi}$.
\end{theorem}

Theorems \ref{adal0} and \ref{adal00} are based on a third main result, in which we give an estimate for  $d_K(f,\mathcal{P}^\psi_{t})$ for fixed $t$.
\begin{theorem}\label{adal3}
	Let $\psi$ be a psh exhaustion on $X$ such that $\frac{i}{2}\partial \bar{\partial}e^\psi$ is smooth and strictly positive outside a compact set $S$ and assume $\psi$ induces an integral estimate for holomorphic functions.  Let $K\subset X$ be compact and $\varphi$ be a continuous psh function on $X$ satisfying:
	\begin{enumerate}[(i)]
		\item There is a constant $t_0>0$ such that $t_0 i\partial \bar{\partial}\varphi+\operatorname{Ricci}\left(\frac{i}{2}\partial \bar{\partial} e^\psi\right)\geq 0$ on $X\setminus S$ and $i\partial\bar{\partial}\varphi>0$ on $S$,
		\item $\varphi\in \mathcal{L}_\psi$,
		\item $\varphi|_K\leq 0$.
	\end{enumerate}
	Let $L>1$, $f$ be a function holomorphic on the set $\{z\in X;\; \varphi(z)<\log(L) \}$ and $\epsilon\in ]0,(L-1)/2[$. Then there are constants $l,T_0$ neither depending on $f$ nor $L$ and a constant $M$ not depending on $f$, such that for any $t\geq T_0$ we have
	\begin{align}\label{tyu}
		d_K(f,\mathcal{P}^\psi_{t})\leq M\|f\|_{\{\varphi\leq \log(L-\epsilon/2)\}}\left(\frac{1+\epsilon}{L-\epsilon}\right)^{t-l}.
	\end{align}
	If $L$ is large enough we can write $M=M_0\|\bar{\partial}\chi\|_{L^2(X\setminus S)}$
	where $M_0$ is a constant neither depending on $f$ nor $L$ and $\chi:X\to\R$ is any $C^\infty$ cutoff function with $\chi=1$ on $\{\varphi<\log(L-\epsilon) \}$ and $\chi=0$ on $\{\varphi>\log(L-\epsilon/2)\}$. Here $\|\bar{\partial}\chi\|_{L^2(X\setminus S)}$ denotes the $L^2$ norm of $\bar{\partial}\chi$ on $X\setminus S$ with respect to the measure $\left(\frac{i}{2}\partial \bar{\partial}e^\psi\right)^N$ and the natural norm on $\Lambda^{0,1}T^*_X$ induced by the metric $\frac{i}{2}\partial \bar{\partial}e^\psi$.
\end{theorem}

\section{Examples}
We start this section by proving two propositions. We then apply them to construct a few examples of psh exhaustion functions $\psi$ on Stein manifolds which are Ricci compensable and induce an integral estimate for holomorphic functions.

\begin{proposition}\label{prop25}
	Let $\psi$ be a psh exhaustion on $X$ such that $\frac{i}{2}\partial \bar{\partial}e^\psi$ is smooth and strictly positive outside a compact set $S$. If there exist functions $\epsilon,M:X\to \R_+$ and constants $A,B>0$ satisfying
	\begin{align*}
		\epsilon(z)\geq e^{-(A\psi^+(z)+B)}\;\;\text{and}\;\;  M(z)\leq e^{A\psi^+(z)+B},\qquad z\in X,
	\end{align*}
	such that for every $z\in X\setminus S$, there is a coordinate patch $\xi: B(\zeta,\epsilon(z))\to X$ with $\xi(\zeta)=z$, $\xi^*\left(\frac{i}{2}\partial \bar{\partial}e^\psi\right)\leq M(z)\omega_0$ and $\omega_0^N\leq M(z)\left(\xi^*\frac{i}{2}\partial \bar{\partial}e^\psi\right)^N$ on $B(\zeta,\epsilon(z))$, then $\psi$ induces an integral estimate for holomorphic functions. 
\end{proposition} 
\begin{proof}
	Let $\delta>0$ and write $r(z)=\delta_1 \min\{\epsilon(z),(M(z))^{-1} \}$ where $\delta_1=\min\{1,\delta\}$. Since  $\xi^*(\frac{i}{2}\partial \bar{\partial}e^\psi)\leq M(z)\omega_0$ on $B(0,\epsilon(z))$ we have
	\begin{align*}
		\xi(B(0,r(z)))\subset B(z,\delta_1,\tfrac{i}{2}\partial \bar{\partial}e^\psi)\subset B(z,\delta,\tfrac{i}{2}\partial \bar{\partial}e^\psi).
	\end{align*}
	Let $F$ be a function holomorphic in a neighborhood of the ball $\overline{B}(z,\delta,\frac{i}{2}\partial \bar{\partial}e^\psi)$ and denote by $v_{2N}$ the volume of the unit ball of dimension $2N$. Then by the sub-mean-value inequality on $\C^N$ we have
	\begin{align*}
		|F(z)|^2&=|F\circ \xi(0)|^2\leq \frac{1}{v_{2N}(r(z))^{2N}}\int_{B(0,r(z))} |F\circ \xi|^2\omega_0^N\\
		&\leq \frac{\max\{\epsilon^{-1}(z),M(z) \}^{2N}}{v_{2N}\delta_1^{2N}}\int_{B(0,r(z))} |F\circ \xi|^2M(z)\left(\xi^*\tfrac{i}{2}\partial \bar{\partial}e^\psi\right)^N\\
		&\leq \frac{\max\{\epsilon^{-1}(z),M(z) \}^{2N}M(z)}{v_{2N}\delta_1^{2N}}\int_{B(z,\delta,\tfrac{i}{2}\partial \bar{\partial}e^\psi)}F(\tfrac{i}{2}\partial \bar{\partial}e^\psi)^N.
	\end{align*}
	By assumption of the growth rate of $\epsilon$ and $M$ we get the result.
\end{proof}

\begin{proposition}\label{prop35}
	Assume $\psi$ is an exhaustion function of the form
	\begin{align*}
		\psi(z)=\log(|g_1(z)|^2+...+|g_m(z)|^2),\qquad z\in X,
	\end{align*}
	for some holomorphic functions $g_1,...,g_m\in \mathcal{O}(X)$.
	\begin{enumerate}[(i)]
		\item The Ricci curvature of $\frac{i}{2}\partial \bar{\partial}e^{\psi}$ is given with
		\begin{align*}
			\operatorname{Ricci}\big(\tfrac{i}{2}\partial \bar{\partial}e^{\psi}\big)= -i\partial \bar{\partial}\log\left( \sum_{1\leq j_1<...< j_N\leq m}|\Det \left(\operatorname{Jac}(g_{j_1},...,g_{j_N})\right)|^2\right).
		\end{align*}
		where the sum is taken over every subcollection $\{j_1\leq...\leq j_N\}\subset \{1,...,m\}$ of size $N$ and $\operatorname{Jac}(g_{j_1},...,g_{j_N})$ is the Jacobian of $g_{j_1},...,g_{j_N}$ with respect to any local coordinate system.
		\medskip
		
		\item If there exist constants $A,B>0$ such that for every $z\in X$ there exists a subcollection $\{g_{j_1},...,g_{j_N}\}\subset \{g_1,...,g_m\}$ and a neighborhood $V$ of $z$ such that $(g_{j_1},...,g_{j_N})$ maps $V$ bijectively to an open ball in $\C^N$ of radius $\epsilon(z)\geq e^{-(A\psi^+(z)+B)}$ and such that for every $g_k\in \{g_1,...,g_m\}\setminus \{g_{j_1},...,g_{j_N}\}$ we have 
		\begin{align}
			i\partial \bar{\partial}|g_k|^2\leq e^{A\psi^++B}i\partial\bar{\partial}\sum_{s=1}^{N}|g_{j_s}|^2\qquad \text{on}\;\; V,\label{wwww}
		\end{align}
		then $\psi$ induces an integral estimate for holomorphic functions.
		
	\end{enumerate}
\end{proposition}

\begin{proof}
	$(i)$ Let $(z_1,...,z_N)$ be some local coordinate chart. We have
	\begin{align*}
		i\partial\bar{\partial}e^{\psi}=\sum_{j,k,r}\frac{\partial g_r}{\partial z_j}\overline{\frac{\partial g_r}{\partial z_k}}dz_j\wedge d\bar{z}_k
	\end{align*} 
	and therefore, by the Cauchy-Binet formula, we have
	\begin{align*}
		(i\partial \bar{\partial}e^{\psi})^N&=\Det\left(\sum_{r}\frac{\partial g_r}{\partial z_j}\overline{\frac{\partial g_r}{\partial z_k}} \right)_{j,k}dV\\
		&=\left(\sum_{1\leq j_1<...< j_N\leq m}|\Det \left(\operatorname{Jac}(f_{j_1},...,f_{j_N})\right)|^2\right)dV.
	\end{align*}
	The result now follows from definition of the Ricci curvature.
	
	$(ii)$ This is just a special case of Proposition \ref{prop25} where the coordinate patch $\xi$ is defined as the inverse of the map $z\to(g_{j_1}(z),...,g_{j_N}(z))$.
\end{proof}

Now we apply propositions \ref{prop25} and \ref{prop35} to construct a few examples where our main results can be applied.

\textbf{\large{Polynomials in $\C^N$}.} Let $X=\C^N$ and let $g_1,...,g_m$ be polynomials on $\C^N$ s.t.\
\begin{align*}
	\psi(z):=\log(|g_1(z)|^2+...+|g_m(z)|^2),\qquad z\in \C^N
\end{align*}
is an exhaustion function. Further assume that the Jacobian of the map $z\to (g_1(z),...,g_m(z))$ has full rank on $\C^N$. Then $\psi$ is Ricci compensable and induces an integral estimate for holomorphic functions. We do not prove this here since this is a special case of our next example.

\textbf{\large{Affine algebraic manifolds}.} Let $X\subset \C^M$ be a non-singular algebraic manifold of dimension $N$ and let $g_1,...,g_m$ be polynomials on $X$ (i.e.\ each $g_j$ is the restriction of a polynomial on $\C^M$ to $X$). Assume that the function
\begin{align*}
	\psi(z)=\log(|g_1(z)|^2+...+|g_m(z)|^2),\qquad z\in X
\end{align*}
is an exhaustion function, and further assume that the Jacobian of the map $X\to \C^m$, $z\to (g_1(z),...,g_m(z))$ has full rank on $X$. Then $\psi$ is Ricci compensable and induces an integral estimate for holomorphic functions. By Rudin \cite{Rud:1968}, after a linear change of variables, we can assume that $X$ is a subset of
\begin{align*}
	\{z=(z_1,...,z_N,z_{N+1},...,z_M)=(z',z'')\in \C^M;\; \|z''\|\leq A(1+\|z'\|)^{B} \}
\end{align*}
for some positive constants $A,B$. This implies that the function $\tau:=\log\|z'\|$ is a special exhaustion function on $X$. Since $\psi\in \mathcal{L}_{C\tau}$ for $C>0$ large enough, we see that the polynomial spaces $\mathcal{P}^\psi_{t}$ have finite dimension.

We now prove that $\psi$ is Ricci compensable and induces an integral estimate for holomorphic functions. Our method is based on Demailly's calculations from the proof of \cite[Proposition 10.1]{Dem:1985}. Indeed, we generalize this result by calculating the Ricci curvature of $\frac{i}{2}\partial \bar{\partial}e^\psi$. 
	
	Let $P_1,...,P_r$ be generators of the ideal $I(X)$ of polynomials in $\C^M$ vanishing on $X$ and let $s=M-N$ be the codimension of $X$. For each $K=\{k_1<...<k_s\}\subset \{1,...,r\}$ and each $L=\{l_1<...<l_N \}\subset \{1,...,m \}$ denote by $J_{K,L}$ the determinant of the Jacobian of the functions $g_{l_1},...,g_{l_N},P_{k_1},...,P_{k_s}$ on $\C^M$. Further write
	\begin{align*}
		U_K:=X\cap \{z\in \C^M:\; dP_{k_1}\wedge...\wedge dP_{k_s}(z)\not=0 \}.
	\end{align*}
	The sets $U_K$ form an open cover of $X$ since it is non-singular. Denote by $(z_1,...,z_M)$ the standard coordinates on $\C^M$, write $T_0:=\{N+1,N+2,...,M \}$ and let $w\in U_K$ be fixed. Without loss of generality we can assume that $X$ can be parameterized in the variables $(z_1,...,z_N)$ in a neighborhood of $w$. That means that in a neighborhood of $w$ in $\C^M$ we have
	\begin{align*}
		|D_{K,T_{0}}|^2:=\left|\Det\left(\tfrac{\partial P_{k}}{\partial z_t } \right)_{k\in K, t\in T_0 } \right|^2\not=0
	\end{align*}
	and therefore
	\begin{align}
		&dg_{l_1}\wedge d\bar{g}_{l_1}\wedge...\wedge dg_{l_N}\wedge d\bar{g}_{l_N}\wedge dP_{k_1}\wedge d\overline{P}_{k_1}\wedge...\wedge dP_{k_s}\wedge d\overline{P}_{k_s}\nonumber\\
		=&|J_{K,L}|^2dz_1\wedge d\bar{z}_1\wedge...\wedge dz_{M}\wedge d\bar{z}_M\label{xxxx}\\
		=&\frac{|J_{K,L}|^2}{|D_{K,T_0}|^2}dz_1\wedge d\bar{z}_1\wedge...\wedge dz_N\wedge d\bar{z}_N\wedge dP_{k_1}\wedge d\overline{P}_{k_1}\wedge...\wedge dP_{k_s}\wedge d\overline{P}_{k_s}.\nonumber
	\end{align} 
	Since the gradients $\nabla P_k$ are orthogonal to the tangent space of $X$ we see that when we restrict the forms from equation $(\ref{xxxx})$ to the submanifold $X$ we have
	\begin{align*}
		dg_{l_1}\wedge d\bar{g}_{l_1}\wedge...\wedge dg_{l_N}\wedge d\bar{g}_{l_N}=\frac{|J_{K,L}|^2}{|D_{K,T_0}|^2}dz_1\wedge d\bar{z}_1\wedge...\wedge dz_N\wedge d\bar{z}_N.
	\end{align*}
	Now by applying Proposition \ref{prop35} $(i)$, and by noticing that the function $\log|D_{K,T_0}|^2$ is pluriharmonic in a neighborhood of $w$, we see that
	\begin{align}
		\operatorname{Ricci}\left(\tfrac{i}{2}\partial\bar{\partial}e^{\psi}\right)=-i\partial\bar{\partial}\log \sum_{|L|=N}|J_{K,L}|^2,\qquad \text{on}\;\; U_K.\label{ricci1}
	\end{align}
	Now let $K_0$ be fixed. For any $K\not= K_0$ the function
	\begin{align*}
		a_{K,K_0}:=\log \sum_{|L|=N}|J_{K,L}|^2-\log \sum_{|L|=N}|J_{K_0,L}|^2
	\end{align*}
	is pluriharmonic on $U_K\cap U_{K_0}$ since it is the difference of two local potentials of the Ricci curvature. Moreover, this function is locally bounded from above on $U_{K_0}$ and since $U_{K_0}\setminus U_K$ is an analytic subset of $U_{K_0}$ the function $a_{K,K_0}$ is psh on $U_{K_0}$. This is true for all $K$ so the function
	\begin{align}
		\log \sum_{|K|=s}e^{a_{K,K_0}}=\log \sum_{|K|=s,|L|=N}|J_{K,L}|^2-\log \sum_{|L|=N}|J_{K_0,L}|^2\label{ricci2}
	\end{align}
	is plurisubharmonic on $U_{K_0}$. Now define the function
	\begin{align*}
		\theta=\log \sum_{|K|=s,|L|=N} |J_{K,L}|^2,\qquad \text{on}\;\; X.
	\end{align*}
	By equations (\ref{ricci1}) and (\ref{ricci2}) we see that $i\partial\bar{\partial}\theta+\operatorname{Ricci}(\frac{i}{2}\partial\bar{\partial}e^{\psi})\geq 0$ on $X$. Moreover, since the Jacobian of $g_1,...,g_m$ has full rank on $X$, the functions $|J_{K,L}|^2$ never vanish at the same time on $X$, i.e.\ we have $\theta>-\infty$ on $X$. By a simple application of Hilbert's Nullstellensatz we can see that there exist constants $A$ and $B$ such that $|\theta|\leq A\psi^++B$ on $X$ and therefore $\theta$ is a Ricci compensator for $\psi$. 
	
	By applying Hilbert's Nullstellensatz again we can find constants $A_1,B_1$ such that for each $z\in X$ we can find $K=\{k_1,...,k_s \}$ and $L=\{l_1,...,l_N \}$ such that 
	\begin{align*}
		|J_{K,L}(z)|^2\geq e^{-(A_1\psi^+(z)+B_1)}.
	\end{align*}
	Since the derivatives of $|J_{K,L}|^2$ have polynomial growth it is simple to show that $(g_{l_1},...,g_{l_N},P_{k_1},...,P_{k_s})$ maps a neighborhood of $z$ in $\C^M$ bijectively to an open ball in $\C^M$ of radius $\epsilon(z)\geq e^{-(A_2\psi^+(z)+B_2)}$ where $A_2,B_2$ are constants independent of $z$. Since the functions $P_{k_1},...,P_{k_s}$ vanish on $X$ we see that $(g_{l_1},...,g_{l_N})$ maps a neighborhood of $z$ in $X$ to an open ball in $\C^N$ of radius $\epsilon(z)$. Since the functions $g_{j}$ are polynomials it is easy to see that inequality $(\ref{wwww})$ is satisfied for some $A,B$.

\textbf{\large{The complex torus}.} Let $X=\C^N/\Z^N$ be the complex torus and
\begin{align*}
	\psi(z)=\|\operatorname{Im}(z)\|=\|y\|,\qquad z=x+iy\in X.
\end{align*}
Then $\psi$ is Ricci compensable and induces an integral estimate for holomorphic functions. The function $\psi$ is itself a special exhaustion function so the polynomial spaces $\mathcal{P}^\psi_{t}$ have finite dimension. Indeed the functions
\begin{align}
	\xi_a(z):=e^{2\pi i\langle z,a\rangle},\qquad a\in \Z^N,\;\; \|a\|\leq \tfrac{t}{2\pi},\;\; z\in X,\label{fourier}
\end{align}
form a basis for $\mathcal{P}^\psi_{t}$ for every $t$. In this case we can apply our main theorems to prove classical results from Fourier analysis.

	We now prove these statements. It is simple to show that $(i\partial \bar{\partial}\|y\|)^N=0$ if $y\not=0$ so $\psi$ is a special exhaustion function. Now suppose $p\in \mathcal{P}_t^\psi$ for some $t$. Since $p$ is periodic it is of the form $p(z)=\sum_{a\in \Z^N} c_a e^{2\pi i \langle z,a\rangle}$ for some constants $c_a$. By the Paley-Wiener theorem we see that $p$ is the Fourier-transform of a distribution with support on $B(0,t)\subset \R^N$. Therefore we have $c_a=0$ if $\|a\|>\frac{t}{2\pi}$ and the functions from (\ref{fourier}) form a basis for $\mathcal{P}_t^\psi$.

 	If $\|y\|\geq 1$, then the largest eigenvalue of the metric
	\begin{align*}
		\frac{i}{2}\partial \bar{\partial}e^{\psi}=\frac{ie^\psi}{8}\sum_{1\leq j,k\leq N}\left(\frac{\delta_{j,k}}{\|y\|}+\frac{y_jy_k(\|y\|-1)}{\|y\|^3} \right) dz_j\wedge d\bar{z}_k
	\end{align*}
	is $\lambda_1=e^\psi/4$ and corresponds to the eigenvector $y$. Therefore we can apply Proposition \ref{prop25} with $\epsilon(z)= \frac{1}{2}$ (we can map the torus $X$ to a strip in $\C^N$ of width one centered at $z$) and $M(z)=\frac{1}{4}e^{\psi(z)+\frac{1}{2}}$ so $\psi$ induces an integral estimate for holomorphic functions. The metric $\frac{i}{2}\partial \bar{\partial}e^{\psi}$ has one more eigenvalue $\lambda_2=\frac{e^\psi}{4\|y\|}$ which corresponds to the $(N-1)$-dimensional eigenspace of vectors perpendicular to $y$. Therefore we have
	\begin{align*}
		\operatorname{Ricci}\left(\tfrac{i}{2}\partial\bar{\partial}e^{\psi} \right)=-i\partial \bar{\partial}\log (\lambda_1\lambda_2^{N-1})=-Ni\partial \bar{\partial}\psi+(N-1)i\partial \bar{\partial}\log \psi.
	\end{align*}
	As a Ricci compensator we can take the function
	\begin{align*}
		\theta=v+\begin{cases}
			N\psi-(N-1)\log \psi,\qquad \text{if}\; \psi\geq 2\\
			2N-(N-1)\log 2,\qquad \text{if}\; \psi<2
		\end{cases}
	\end{align*}
where $v\in \mathcal{L}_{\psi}$ is any function which is strictly psh on \{$\psi<2\}$. It is a simple exercise to show that $\theta$ is indeed psh.

\textbf{The complement of a graph of a holomorphic function.} Let $f:\C^{N-1}\to \C$ be a holomorphic function in $(N-1)$ variables and define
\begin{align*}
	F(z):=z_N-f(z'),\qquad z=(z',z_N)\in \C^N.
\end{align*}
Let $X=\C^N\setminus \{F=0 \}$ and
\begin{align*}
	\psi(z):=\log\left(\|z'\|^2+|F(z)|^2+|F^{-1}(z)|^2\right),\qquad z\in X.
\end{align*}
The function $\psi$ is Ricci compensable and induces an integral estimate for holomorphic functions. In \cite[Theorem 4.1]{Ayt:2014} we see that $\tau:=\log(\|z'\|+|F(z)-1|^2)-\log|F(z)|$ is a special exhaustion function. It is easy to see that $\psi\in\mathcal{L}_{C\tau}$ for $C>0$ large enough and therefore the polynomial spaces $\mathcal{P}^\psi_{t}$ have finite dimension.

	To prove these statements we first observe that applying Proposition \ref{prop35} $(i)$ gives
		\begin{align*}
			\operatorname{Ricci}\left(\tfrac{i}{2}\partial\bar{\partial}e^{\psi} \right)=-i\partial\bar{\partial}\log(1+|F|^{-2})
		\end{align*}
	so we can take $\theta=\log(1+|F|^{-2})$ as a Ricci compensator. The map $z\to (z',F^{-1}(z))$ is a bijection from $X$ to $\C^{N-1}\times\C^*$ so we can take $\{z',F^{-1} \}$ as the subcollection mentioned in Proposition \ref{prop35} $(ii)$ and  $\epsilon(z)=|F(z)|^{-1}$. We just have to check that (\ref{wwww}) from Proposition \ref{prop35} is satisfied. Indeed we have  
	\begin{align*}
		i\partial \bar{\partial}|F|^2= |F|^{4}i\partial \bar{\partial}|F|^{-2}\qquad \text{on}\;\; X
	\end{align*}
 	so $\psi$ induces an integral estimate for holomorphic functions.
\section{Proofs}
In this section we prove all results from Section 3. Recall that $X$ is always a Stein manifold and $\psi$ is a psh exhaustion function on $X$. We start with auxiliary propositions.
\begin{proposition}\label{key}
	Assume that $\frac{i}{2}\partial \bar{\partial}e^\psi$ is smooth and strictly positive outside a compact subset $S$ of $X$. Then there exists a Kähler form $\omega$ on $X$ such that $\omega=\frac{i}{2}\partial \bar{\partial}e^\psi$ outside a compact subset of $X$. Moreover, the metric $\omega$ is complete.
\end{proposition}
\begin{proof}
	By adding a constant to $\psi$ we can assume that $\psi|_S\leq -1$. We define
	\[\omega:=\tfrac{i}{2}\partial \bar{\partial}\left(e^{\Gamma\circ \psi}+\epsilon \chi u\right) \]
	where $\Gamma:\R\to \R$ is a smooth, increasing and convex function satisfying $\Gamma(x)=-1/2$ for $x\leq -1$ and $\Gamma(x)=x$ for $x>0$, $u$ is any strictly psh function defined in a neighborhood of $\{\psi\leq 0 \}$, $\chi:X\to [0,1]$ is a smooth function with $\chi=1$ on $\{\psi\leq -1/2\}$ and $\chi=0$ on $\{\psi\geq -1/4 \}$ and $\epsilon$ is a small constant. We clearly have $\omega=\frac{i}{2}\partial \bar{\partial}e^\psi$ on $\{\psi>0\}$ and $\omega$ is strictly positive on $\{\psi\leq -1/2 \}\cup\{\psi\geq -1/4 \}$. The form $\frac{i}{2}\partial\bar{\partial}e^{\Gamma\circ \psi}$ is strictly positive on $\{-1/2\leq \psi\leq -1/4 \}$ so by choosing $\epsilon$ small enough we can make $\omega$ strictly positive on $X$.
	
	On $X\setminus\{\psi\leq 0\}$ we have $\omega=\tfrac{i}{2}e^{\psi}\left( \partial \psi\wedge \bar{\partial}\psi+\partial\bar{\partial}\psi\right)\geq \tfrac{i}{2}\partial \psi\wedge \bar{\partial}\psi$ and therefore $|d\psi|_{\omega}=|\partial\psi+\bar{\partial}\psi|_{\omega}\leq 2|\partial\psi|_{\omega}\leq 2|\partial \psi|_{\partial \psi\wedge \bar{\partial}\psi}= 2.$
		By \cite[Lemma VIII 2.4]{Dem:2012} the metric is complete. 
\end{proof}
For the rest of the section $\omega$ is a Kähler form as described in Proposition \ref{key}. By adding a constant to $\psi$ we can always assume that $S=\{\psi\leq 0\}$ and $\omega=\frac{i}{2}\partial \bar{\partial}e^{\psi}$ on $X\setminus S$. If $v$ is a tangent vector in $X$ we denote by $\omega(v)$ the length of $v$ with respect to the metric $\omega$. Whenever we work in local coordinates $(z_1,...,z_N)$ we write
\begin{align*}
	\nabla g=\left(\frac{\partial g}{\partial z_1},...,\frac{\partial g}{\partial z_N}\right)\qquad \text{and}\qquad \langle v,w\rangle=\sum_{j=1}^{N}v_jw_j
\end{align*}
for every function $g\in C^1$ and tangent vectors $v,w$.
\begin{lemma}\label{l52}
	If $B(z,r,\omega)\cap S=\emptyset$ then $|e^{\psi(\zeta)/2}-e^{\psi(z)/2}|\leq \frac{r}{2}$ for every $\zeta\in B(z,r,\omega)$.
\end{lemma}
\begin{proof}
	For every tangent vector $v$ we have
	\begin{align}
		\omega(v)=\frac{i}{2}e^{\psi}(\partial\bar{\partial}\psi+\partial\psi\wedge\bar{\partial}\psi)(v)\geq e^{\psi}|\langle \nabla\psi,v\rangle|^2\label{a}
	\end{align}
	Let $\gamma:[0,1]\to X$ be a geodesic from $z$ to $\zeta$ of length $r$. By the fundamental theorem of calculus, and by (\ref{a}) we have
	\begin{align*}
		e^{\psi(\zeta)/2}-e^{\psi(z)/2}=\int_{0}^{1}\frac{\partial}{\partial t}&e^{(\psi\circ\gamma)/2}dt\\
		&\leq\frac{1}{2}\int_{0}^{1}e^{(\psi\circ\gamma)/2}|\langle \nabla \psi,\gamma'\rangle|dt\leq \frac{1}{2}\int_{\gamma} \omega^{\frac{1}{2}}=\frac{r}{2}.
	\end{align*}
\end{proof}

 The following theorem is a special case of a famous result by Skoda \cite{Skoda:1977}.
\begin{theorem}
	\label{hormander}
	Let $\varphi$ be a continuous psh function on $X$ satisfying $\frac{i}{2}\partial\bar{\partial}\varphi+\operatorname{Ricci}(\omega)\geq 0$ on $X$. Then, for any $(0,q)$-form $f$ with $C^\infty$ (resp. $L^2_{\operatorname{loc}}$) coefficients satisfying $\bar{\partial}f=0$ and $\int_{X}|f|_\omega^2e^{-\varphi}\omega^N<\infty$ there is a $(0,q-1)$-form $u$ with $C^\infty$ (resp. $L^2_{\operatorname{loc}}$) coefficients such that $\bar{\partial}u=f$ and
	\begin{align*}
		\int_{X}|u|_{\omega}^2(e^{\psi}+1)^{-2}e^{-2\varphi}\omega^N\leq \frac{1}{2N}\int_X s|f|_{\omega}^2e^{-2\varphi}\omega^N,
	\end{align*}
	where $s$ is a non-negative function on $X$ which equals $1$ on $X\setminus S$.
\end{theorem}
\begin{proof}
	First assume that $\varphi$ is $C^{\infty}$. Consider the line bundle $E:=X\times \C$ over $X$ with the trivial projection. On the fibers of $E$ we define the Hermitian product
	\begin{align*}
		\langle \zeta_1,\zeta_2\rangle_z :=(1+e^{\psi(z)})^{-2}\zeta_1\bar{\zeta}_2
	\end{align*}
	and denote by $|\cdot|^2_E=(1+e^{\psi(z)})^{-2}|\cdot|^2$ the corresponding norm. Denote by $i\Theta(E)=2i\partial \bar{\partial}\log(1+e^{\psi})$ the Chern curvature tensor on $E$ with respect to this metric (see for example comment (12.6) Chapter V in \cite{Dem:2012}). On $X\setminus S$ we have
	\begin{align*}
		i(1+e^\psi)^2\Theta(E)=4\omega+4e^{2\psi}i\partial\bar{\partial}\psi\geq4\omega
	\end{align*}
	so, by assumption on $\varphi$, we have
	\begin{align}\label{eigingildi}
		i\Theta(E)+\operatorname{Ricci}(\omega)+i\partial\bar{\partial}\varphi\geq\frac{4\omega}{(1+e^\psi)^2}
	\end{align}
	on $X\setminus S$. Therefore the sum of the eigenvalues of the $(1,1)$-form on the left hand side of (\ref{eigingildi}) with respect to the metric $\omega$ is larger than or equal to $4N(1+e^\psi)^{-2}$.
	
	Now, if we consider $f$ as a section of the cotangent bundle of $E$, by \cite[Theorem VIII 6.5]{Dem:2012} we can find a $(0,q-1)$ form $u$ such that $\bar{\partial}u=f$ and
	\begin{align*}
		\int_{X}|u|_E^2e^{-2\varphi}\omega^N\leq \frac{1}{4N}\int_X s(1+e^\psi)^2|f|_E^2e^{-2\varphi}\omega^N,
	\end{align*}
	where $s$ is a positive function equal to $1$ on $X\setminus S$. If we replace $|\cdot|_E^2$ by $(1+e^{\psi(z)})^{-2}|\cdot|^2$ the result follows.
	
	If $\varphi$ is not $C^{\infty}$ we get the result by finding a decreasing sequence $(\varphi_n)_{n\in \N}$ of $C^{\infty}$ psh functions such that $\varphi_n\searrow \varphi$ in the $L^1_{\text{loc}}$ topology and by taking the limit. Since $X$ is Stein, such a sequence exists.
\end{proof}

\begin{proof}[Proof of Theorem \ref{adal3}.]
By assumption $(i)$ and compactness of $S$ we can find a constant $T_0$ such that $T_0\frac{i}{2}\partial\bar{\partial}\varphi+\operatorname{Ricci}(\omega)\geq 0$. Let $\chi:X\to \R $ be a $C^\infty$ cutoff function with $\chi=1$ on $\{\varphi<\log(L-\epsilon) \}$ and $\chi=0$ on $\{ \varphi>\log(L-\epsilon/2)\}$. For any $t>T_0$ we can apply Theorem \ref{hormander} to find a function $u_t$ solving the $\overline{\partial}$-equation  $\bar{\partial}u_t=\bar{\partial}(f\chi)=f\bar{\partial}\chi$ on $X$ and satisfying
	\begin{align}
		\int_{X}|u_t|^2(e^{\psi}+1)^{-2}e^{-2t\varphi}\omega^N\leq\frac{1}{2N}\int_Xs|f\bar{\partial}\chi|_{\omega}^2e^{-2t\varphi}\omega^N.\label{j}
	\end{align}
	Now we define the function $p_t:=\chi f-u_t$. It is clear that $p_t$ is an entire function.
	Let $z\in X$ be such that $B(z,1,\omega)$ does not intersect $S$ or the support of $\chi$. Then $u_t=p_t$ on  $B(z,1,\omega)$.
	
	 Since $\psi$ induces an integral estimate for holomorphic functions (Def \ref{skil} (ii)) we can find constants $A$ and $B$ (not depending on $z$ or $t$) such that
	\begin{align*}
		|p_t&(z)|^2\leq e^{A\psi^+(z)+B}\int_{B(z,1,\omega)}|u_t|^2\omega^N,\nonumber\\
		&\leq e^{A\psi^+(z)+B}\sup_{\zeta\in B(z,1,\omega)}\left(e^{2t\varphi(\zeta)}(e^{\psi(\zeta)}+1)^2\right)\int_{X}|u_t|^2(e^{\psi}+1)^{-2}e^{-2t\varphi}\omega^N.
	\end{align*}
	Since $\varphi\in \mathcal{L}_\psi$ we can now use inequality (\ref{j}) and Lemma \ref{l52} to show that $p_t\in \mathcal{P}^\psi_{l+t}$ where $l=\frac{A}{2}+1$.
	
	Let $r>0$ be small enough such that $B(z,r,\omega)\subset\subset \{\varphi<\log(1+\epsilon)\}\subset \{\chi=1\}$ for all $z\in K$.  Then $u_t\in\mathcal{O}(\overline{B}(z,r,\omega))$ for $z\in K$. Since $\psi$ induces an integral estimate for holomorphic functions, and by the compactness of the sets $S$ and $K$ we can find a positive constant $C$ such that for every $z\in K$ we have
	\begin{align}
		|f(z)-p_t(z)|^2&=|u_t(z)|^2\leq C\left(\sup_{\zeta\in B(z,r,\omega)}1+e^{\psi(\zeta)}\right)^{-2}\int_{B(z,r,\omega)}|u_t|^2\omega^N\nonumber\\
		&\leq C\left(\sup_{\zeta\in B(z,r,\omega)}e^{2t\varphi(\zeta)}\right)\int_{X}|u_t|^2(e^{\psi}+1)^{-2}e^{-2t\varphi}\omega^N.\label{qqqqq}
	\end{align}
	By the definition of $r$ we can estimate $\sup_{\zeta\in B(z,r,\omega)}e^{2t\varphi(\zeta)}$ with $(1+\epsilon)^{2t}$ and recall that we have $\log(L-\epsilon/2)>\varphi>\log(L-\epsilon)$ on the support of $\bar{\partial}\chi$. By inequalities (\ref{j}) and (\ref{qqqqq}) we therefore have
	\begin{align*}
		|f(z)-p_t(z)|^2&\leq C(1+\epsilon)^{2t}\int_{X}s|f\bar{\partial}\chi|_{\omega}^2e^{-2t\varphi}\omega^N\\
		&\leq C\|s\|_X\left(\frac{1+\epsilon}{L-\epsilon}\right)^{2t}\|f\|^2_{\{\varphi<\log(L-\epsilon/2) \}}\|\bar{\partial}\chi\|^2_{L^2(X)},\qquad z\in K
	\end{align*}
	where $\|s\|_X$ is the sup-norm of $s$ on $X$ and $\|\bar{\partial}\chi\|^2_{L^2(X)}$ is the $L^2$ norm of $\bar{\partial}\chi$ with respect to the metric $\omega$ and measure $\omega^N$. We have already seen that $p_t\in \mathcal{P}^\psi_{l+t}$ and by replacing $t$ with $t-l$ we have the result with $M=C\|s\|_X\|\bar{\partial}\chi\|_{L^2(X)}$. If $L$ is large enough the support of $\bar{\partial}\chi$ does not intersect $S$ and $\|\bar{\partial}\chi\|_{L^2(X)}=\|\bar{\partial}\chi\|_{L^2(X\setminus S)}$.
\end{proof}
\begin{lemma}\label{vvvv}
	Let $L,\epsilon>0$, $\varphi\in \mathcal{L}_\psi^+$ be continuous and $\theta$ be a continuous function satisfying the growth condition
	\begin{align*}
		|\theta(z)|\leq A\psi^+(z)+B,\qquad z\in X
	\end{align*} 
	for some constants $A,B$. Write $\tilde{\varphi}_t:=(1-t^{-1})\varphi+t^{-1}\theta$. Then there exists $T$ such that for $t>T$ we have
	\begin{align}
		\{z\in X;\; \tilde{\varphi}_t(z)<L-\epsilon \}\subset \{z\in X;\;\varphi(z)<L \},\label{p}
	\end{align}
	and the function $\tilde{\varphi}_{t}$ is an exhaustion function. 
\end{lemma}
\begin{proof}
	It is trivial to check that 
	\begin{align*}
		\{z\in X;\; \theta(z)-\varphi(z)\geq 0 \}\cap 	\{z\in X;\; \tilde{\varphi}_t(z)<L-\epsilon \}\subset \{z\in X;\;\varphi(z)<L  \}
	\end{align*}
	for every $t>0$. Therefore we only need to show that
	\begin{align*}
		\{z\in X;\; \theta(z)-\varphi(z)<0 \}\cap 	\{z\in X;\; \tilde{\varphi}_t(z)<L-\epsilon \}\subset \{z\in X;\;\varphi(z)<L  \}
	\end{align*}
	for large enough $t$.
	
	 Since $\varphi\in \mathcal{L}_\psi^+$ and by $(\ref{p})$ we can find positive constants $C_1$ and $C_2$ such that
	\begin{align*}
		|\theta(z)|\leq C_1\varphi^+(z)+C_2,\qquad z\in X.
	\end{align*}
	If $t>2$ and $z$ is such that $\varphi(z)>1$, then we have
	\begin{align}
		\tilde{\varphi}_t(z)=(1-t^{-1})\varphi(z)+&t^{-1}\theta(z)>\frac{1}{2}\left(\varphi(z)-\frac{2}{t}|\theta(z)|\right)\nonumber\\
		&\geq \frac{\varphi(z)}{2}\left(1-\frac{2C_1}{t}\right)-\frac{C_2}{t}.\label{v}
	\end{align}
	Since $\varphi$ is an exhaustion, inequality (\ref{v}) implies that $\tilde{\varphi}_t$ is also an exhaustion for $t\geq T_0:=\max\{2C_1+1,2\}$. In particular the set $\{\tilde{\varphi}_{T_0}\leq L-\epsilon \}$ is compact so the continuous function $\theta-\varphi$ has a lower bound $-M<0$ on it. We define $T:=\max\{T_0,\frac{M}{\epsilon} \}$. Now let $t>T$ and $z\in \{\theta-\varphi<0 \}\cap \{\tilde{\varphi}_t<L-\epsilon\}$. Since $\theta(z)-\varphi(z)<0$ and $t>T_0$ we have $\tilde{\varphi}_{T_0}(z)<\tilde{\varphi}_{t}(z)\leq L-\epsilon$ so $\theta(z)-\varphi(z)\geq -M$ and
	\begin{align*}
		\varphi(z)=\tilde{\varphi}_t(z)-t^{-1}(\theta(z)-\varphi(z))<L-\epsilon-\frac{\epsilon}{M}(-M)=L.
	\end{align*}
\end{proof}
\begin{proof}[Proof of Theorem \ref{adal0}]
	Let $\epsilon\in]0,(L-1)/2[$ and let $\theta$ be a Ricci compensator for $\psi$. By adding a constant to $\theta$ we can assume $\theta<0$ on $K$.  Now apply Lemma \ref{vvvv} to find $T>0$ such that $\{\tilde{\varphi}_T<\log(L-\epsilon) \}\subset \{\varphi<\log(L)\}$ where $\tilde{\varphi}_T:=(1-T^{-1})\varphi+T^{-1}\theta$. We can now apply Theorem \ref{adal3} with $\varphi$ replaced by $\tilde{\varphi}_T$ and with $L$ replaced by $L-\epsilon$. For $t$ large enough we have 
	\begin{align*}
		d_K(f,\mathcal{P}^\psi_{t})\leq M\|f\|_{\{\tilde{\varphi}_T\leq \log(L-3\epsilon/2)\}}\left(\frac{1+\epsilon}{L-2\epsilon}\right)^{t-l}.
	\end{align*}
	and
	\begin{align*}
		\limsup_{t\to\infty}(d_K(f,\mathcal{P}^\psi_{t}))^{1/t}\leq \limsup_{t\to\infty}\left(\frac{1+\epsilon}{L-2\epsilon}\right)^{\frac{t-l}{t}}=\frac{1+\epsilon}{L-2\epsilon}.
	\end{align*}
	Since $\epsilon>0$ was arbitrary the result follows.
\end{proof}

\begin{proof}[Proof of Proposition \ref{prop33}]
	By assumption there is a sequence $(p_n)_{n\in\N}$ of functions on $X$ such that $p_n\in \mathcal{P}^\psi_{n}$ for all $n\in \N$ and $\|p_n-f\|_K\leq (L-\epsilon(n))^{-n}$ where $\epsilon:\N\to \R_+$ is a decreasing function satisfying $\lim_{n\to\infty}\epsilon(n)=0$. We claim that the sum $p_1+\sum_{n=1}^{\infty}(p_{n+1}-p_n)$ is uniformly convergent on compact subsets of $\{V_{K,\psi}<\log(L) \}$. Indeed for $l<L$ we have
	\begin{align*}
		&\|p_n(z)-p_{n-1}(z)\|_{\{V_{K,\psi}(z)\leq l \}}\leq\|p_n-p_{n-1}\|_{K} \|e^{V_{K,\psi}}\|_{\{V_{K,\psi}(z)\leq l \}}\\
		&\leq \left(\|p_n-f\|_{K}+\|f-p_{n-1}\|_{K}\right)l^n\leq 2l\left(\frac{l}{L-\epsilon(n-1)}\right)^{n-1}.
	\end{align*}
Since $l<L$ and $\epsilon(n)$ converges to $0$ the series converges. It is obviously equal to $f$ on $K$.
\end{proof}
\begin{lemma}\label{summulemma}
	For $a,\varrho>0$ we have
	\begin{align*}
		\sum_{n=1}^{\infty}\left(\frac{a}{n}\right)^{n/\varrho}\leq 1+2^{\varrho}ae^{\frac{a}{e\varrho}}.
	\end{align*}
\end{lemma}
\begin{proof}
	Since $(a/n)^{n/\varrho}<2^{-n}$ for $n>\lfloor 2^{\varrho}a\rfloor$ we have $\sum_{n=\lfloor2^{\varrho}a\rfloor+1}^{\infty}\left(\frac{a}{n}\right)^{n/\varrho}\leq 1$. The function $x\to (a/x)^{x/\varrho}$ is maximized when $x=a/e$. Therefore
	\begin{align*}
		\sum_{n=1}^{\lfloor 2^{\varrho}a\rfloor}\left(\frac{a}{n}\right)^{n/\varrho}\leq \lfloor 2^{\varrho}a\rfloor \left(\frac{a}{a/e}\right)^{\frac{a/e}{\varrho}}\leq 2^{\varrho}ae^{\frac{a}{e\varrho}}.
	\end{align*}
\end{proof}
\begin{lemma}\label{kilemma}
	Let $\varphi\in\mathcal{L}^+_{\psi}$ and $L_0>1$ be large enough such that $\{\psi< 1\}\subset \{\varphi<\log(L_0) \}$. Then for any constants $L_1,L_2$ with $L_2>L_1>L_0$ there exists a function $\chi\in C^{\infty}(X)$ with $\chi=1$ on $\{\varphi<\log(L_1)\}$, $\chi=0$ on $\{\varphi>\log(L_2) \}$ and $\|\bar{\partial}\chi\|^2_{L^2(X)}\leq \frac{M_1L_2^2}{L_2-L_1}\int_{\{\psi\leq\log(L_2)+M_2 \}}\omega^N$ where $M_1$ and $M_2$ are constants independent of $L_1$ and $L_2$.
\end{lemma}
\begin{proof}
	Let $\chi_0\in C^{\infty}(\R)$ be such that $\chi_0(x)=1$ if $x\leq 1$ and $\chi_0(x)=0$ if $x\geq L_2/L_1$. We can choose such a function such that $\|\chi_0'\|_\R\leq 2(L_2/L_1-1)^{-1}$. Now define
	\begin{align*}
		\chi(z)=\chi_0\left(e^{\varphi(z)}/L_1 \right),\qquad z\in X.
	\end{align*}
	We clearly have $\chi=1$ on $\{\varphi<\log(L_1)\}$ and $\chi=0$ on $\{\varphi>\log(L_2) \}$ so we only have to prove the estimate for $\|\bar{\partial}\chi\|^2_{L^2(X)}$. First notice that
	\begin{align}
		|\bar{\partial}\chi|^2_{\omega}\omega^N&\leq \frac{\|\chi_0'\|_{\R}}{L_1}|\bar{\partial}e^\varphi|_{\omega}^2\omega^N\leq \frac{2}{L_2-L_1}|\bar{\partial}e^\varphi|_{\omega}^2\omega^N\nonumber\\
		&=\frac{2i}{L_2-L_1}\partial e^{\varphi}\wedge \bar{\partial}e^{\varphi}\wedge \omega^{N-1}\leq\frac{i}{L_2-L_1}(\partial\bar{\partial}e^{2\varphi})\wedge\omega^{N-1}\label{h}.
	\end{align}
	By assumption there exists a constant $C$ such that $\varphi^+-C\leq \psi^+\leq \varphi^++C$ on $X$. Let $\Gamma_0\in C^{\infty}(\R)$ be such that $\Gamma_0(x)=1$ if $1\leq x\leq \log(L_2)+C$ and $\Gamma_0=0$ if $x\leq 0$ or $x\geq \log(L_2)+C+1$. The function $\Gamma_0$ can be chosen such that $\max\{\|\Gamma_0'\|_{\R},\|\Gamma_0''\|_{\R}\}\leq 4$. Now define $\Gamma:=\Gamma_0\circ \psi$. Then $\Gamma$ equals $1$ on the support of the $(0,1)$-form $\bar{\partial}\chi$ and the support of $\Gamma$ is a subset of
	\begin{align*}
		F:=\{0\leq \psi(z)\leq \log(L_2)+C+1 \}\subset X.
	\end{align*}
	Recall that we have $\omega=\frac{i}{2}\partial \bar{\partial}e^\psi$ on $F$. By (\ref{h}) we have 
	\begin{align*}
		\|\bar{\partial}\chi\|^2_{L^2(X)}&=\int_{X}\Gamma |\bar{\partial}\chi|^2_{\omega}\omega^N\leq \frac{1}{L_2-L_1}\int_{F}\Gamma (i\partial\bar{\partial}e^{2\varphi})\wedge\omega^{N-1}\\
		&=\frac{1}{L_2-L_1}\int_{F}e^{2\varphi}(i\partial\bar{\partial}\Gamma)\wedge\omega^{N-1}\\
		&\leq\frac{1}{L_2-L_1}\int_{F}4e^{2(\psi+C)}i(\partial \psi\wedge \bar{\partial}\psi+\partial\bar{\partial \psi})\wedge\omega^{N-1}\\
		&\leq  \frac{8L^2_2e^{2C+1}}{L_2-L_1}\int_{F}e^{-\psi}\omega^N\leq \frac{8L^2_2e^{2C+1}}{L_2-L_1}\int_{\{\psi\leq \log(L_2)+C+1 \}}\omega^N.
	\end{align*}
\end{proof}
\begin{proof}[Proof of Theorem \ref{adal00}]
		Assume inequalities (\ref{jafna1}) are true and let $\epsilon>0$. We first consider the case when $\varphi$ satisfies $(i)$ from Theorem \ref{adal3}. By assumption there is a constant $C$ such that
		\begin{align}
			\|f\|_{\{\varphi\leq \log(r)\}}\leq C\exp((\sigma+\epsilon)r^{\varrho}),\qquad r\geq 1.\label{arg}
		\end{align}
		Since $f$ is entire, inequality (\ref{tyu}) from Theorem \ref{adal3} is true for all $L>1$ and all large $t$. In particular if we take $L=L(t):=(t/(\varrho\sigma ))^{\varrho^{-1}}+\epsilon/2$ then by (\ref{tyu}) and (\ref{arg}) we have
		\begin{align}
			d_K(f,\mathcal{P}^\psi_{t})\leq M_0C\exp\left(\frac{(\sigma+\epsilon)t}{\varrho\sigma}\right)\|\bar{\partial}\chi_t\|_{L^2(X)}\left(\frac{1+\epsilon}{L(t)-\epsilon}\right)^{t-l}\label{fff}
		\end{align}
		for every $t$ large enough. Here $\chi_t$ is a cut-off function with $\chi_t=1$ on $\{\varphi<\log(L(t)-\epsilon) \}$ and $\chi_t=0$ on $\{\varphi>\log(L(t)-\epsilon/2) \}$. By Lemma \ref{kilemma} and assumption (\ref{volgrowth}) (taking $r=\varrho/2$) there are constants $M_1,M_2,A,B$ such that 
	\begin{align}
		\|\bar{\partial}\chi_t\|^2_{L^2(X)}&\leq \frac{M_1(L(t)-\epsilon/2)^2}{\epsilon/2}\int_{\{\psi\leq \log(L(t)-\epsilon/2)+M_2\}}\omega^N\nonumber\\
		&\leq \frac{M_1(L(t)-\epsilon/2)^2}{\epsilon/2}\exp\left(Ae^{\frac{M_2\varrho}{2}}(L(t)-\epsilon/2)^{\frac{\varrho}{2}}+B \right)\nonumber\\
		&=\frac{2M_1}{\epsilon}\left( \frac{t}{\varrho\sigma}\right)^{2\varrho^{-1}}\exp\left(Ae^{\frac{M_2\varrho}{2}}\left(\frac{t}{\varrho\sigma} \right)^{\frac{1}{2}}+B \right).\label{ggg}
	\end{align}
	Now, combining (\ref{fff}) and (\ref{ggg}), we get
	\begin{align}
		\limsup_{t\to\infty}t(d_K(f,\mathcal{P}^\psi_{t}))^{\varrho/t}\leq \exp\left(\frac{\sigma+\epsilon}{\sigma}\right)(1+\epsilon)^{\varrho}\varrho\sigma.\label{iii}
	\end{align}
	Since (\ref{iii}) is true for all $\epsilon>0$ the result follows.
	
	Now consider the case when $\varphi$ does not satisfy $(i)$ from Theorem \ref{adal3}. Let $\theta$ be a Ricci compensator for $\psi$ satisfying $\theta|_K\leq 0$, and let $T$ be large enough such that $\tilde{\varphi}_{T}$ (as defined in Lemma \ref{vvvv}) is an exhaustion function. Then clearly the function $\varphi_{\epsilon}:=(1-\epsilon)\varphi+\epsilon\tilde{\varphi}_{T}$ satisfies $(i)-(iii)$ from Theorem \ref{adal3}. Moreover, since $\tilde{\varphi}_{T}$ is an exhaustion, we have
	\begin{align}
		\{\varphi_\epsilon\leq \log(r) \}\subset \{(1-\epsilon)\varphi\leq \log(r) \},\nonumber\\ \text{and}\qquad \|f\|_{\{\varphi_\epsilon\leq \log(r)\}}\leq \|f\|_{\{\varphi\leq \log(r^{1/(1-\epsilon)}) \}}\label{pppp}
	\end{align}
	for $r>1$ large enough. By assumption (\ref{jafna1}) and by (\ref{pppp}) we have
	\begin{align*}
		\limsup_{r\to\infty}\frac{\log^+\log\|f\|_{\{\varphi_\epsilon\leq \log(r)\}}}{\log(r)}\leq\limsup_{r\to\infty}\frac{\log^+\log\|f\|_{\{\varphi\leq \log(r^{1/(1-\epsilon)})\}}}{(1-\epsilon)\log(r^{1/(1-\epsilon)})}\leq \frac{\varrho}{1-\epsilon}
	\end{align*}
	and
	\begin{align*}
		\limsup_{r\to\infty}\frac{\log\|f\|_{\{\varphi_\epsilon\leq \log(r)\}}}{r^{\varrho/(1-\epsilon)}}\leq \limsup_{r\to\infty}\frac{\log\|f\|_{\{\varphi\leq \log(r^{1/(1-\epsilon)})\}}}{(r^{1/(1-\epsilon)})^{\varrho}}\leq \sigma.
	\end{align*}
	We can now apply our previous conclusion with $\varphi$ replaced by $\varphi_{\epsilon}$ and $\varrho$ replaced by $\varrho/(1-\epsilon)$ and we have
	\begin{align*}
		\limsup_{t\to\infty}t(d_K(f,\mathcal{P}^\psi_{t}))^{\varrho/((1-\epsilon)t)}\leq e\frac{\varrho}{1-\epsilon}\sigma.
	\end{align*}
	Since this is true for every $\epsilon>0$ the result follows.

	To prove the converse, assume inequality (\ref{jafna3}) is true with $\varphi$ replaced by $V_{K,\psi}$. Then for $\epsilon>0$ we can find $M$ such that for any $n\geq M$ there is a function $p_n\in \mathcal{P}^\psi_{n}$ such that 
	\begin{align}
		\|f-p_n\|_K\leq\left(\frac{e\varrho\sigma(1+\epsilon)}{n}\right)^{n/\varrho}.\label{qqqq}
	\end{align} 
	Consider the function $G:=p_M+\sum_{n=M}^{\infty}(p_{n+1}-p_n)$. Clearly we have $G=f$ on $K$. Moreover, by (\ref{qqqq}) and Lemma \ref{summulemma}, we have
	\begin{align*}
		|G|&\leq|p_M|+ \sum_{n=M}^{\infty}|p_{n+1}-p_n|\leq \|p_M\|_Ke^{MV_{K,\psi}}+\sum_{n=M}^{\infty}\|p_{n+1}-p_n\|_Ke^{(n+1)V_{K,\psi}}\\
		&\leq \|p_M\|_Ke^{MV_{K,\psi}}+\sum_{n=M}^{\infty}\left(\|p_{n+1}-f\|_K+\|p_n-f\|_K\right)e^{(n+1)V_{K,\psi}}\\
		&\leq  \|p_M\|_Ke^{MV_{K,\psi}}+2e^{V_{K,\psi}}\sum_{n=M}^{\infty}\left(\frac{e\varrho\sigma(1+\epsilon)e^{\varrho V_{K,\psi}}}{n}\right)^{n/\varrho}\\
		&\leq \|p_M\|_Ke^{M V_{K,\psi}}+2e^{V_{K,\psi}}\left(1+2^{\varrho}e\varrho\sigma(1+\epsilon)\exp\left(\varrho V_{K,\psi}+\sigma(1+\epsilon)e^{\varrho V_{K,\psi}} \right)\right).
	\end{align*}
	In particular we have 
	\begin{align*}
		\|G\|_{\{V_{K,\psi}\leq \log(r)\}}\leq \|p_M\|_Kr^{M}+2r(1+2^{\varrho}e\varrho\sigma(1+\epsilon)r^{\varrho}e^{\sigma(1+\epsilon)r^{\varrho}})
	\end{align*}
	and now it is easy to see that
	\begin{align*}
		\limsup_{r\to\infty}\frac{\log\|G\|_{\{\varphi\leq \log(r)\}}}{r^{\varrho}}\leq \sigma(1+\epsilon)
	\end{align*}
	and the second inequality of (\ref{jafna1}) follows by letting $\epsilon\to 0$. By simple calculus we can show that the first inequality of (\ref{jafna1}) follows from the second one.

\end{proof}

\bibliographystyle{plain}
\bibliography{bibref}

\begin{thebibliography}{10}

\bibitem{Ayt:2011}
A.~Aytuna and A.~Sadullaev.
\newblock {$S^*$}-parabolic manifolds.
\newblock {\em TWMS J. Pure Appl. Math.}, 2(1):6--9, 2011.

\bibitem{Ayt:2014}
A.~Aytuna and A.~Sadullaev.
\newblock Parabolic {S}tein manifolds.
\newblock {\em Math. Scand.}, 114(1):86--109, 2014.

\bibitem{Ayt:2015}
Ayd\i~n Aytuna and Azimbay Sadullaev.
\newblock Polynomials on parabolic manifolds.
\newblock In {\em Topics in several complex variables}, volume 662 of {\em
  Contemp. Math.}, pages 1--22. Amer. Math. Soc., Providence, RI, 2016.

\bibitem{Cir:1976:2}
E.~M. {\v{C}}irka.
\newblock Meromorphic continuation, and the rate of rational approximations in
  {$C^{N}$}.
\newblock {\em Mat. Sb. (N.S.)}, 99(141)(4):615--625, 1976.

\bibitem{Cir:1976}
E.~M. {\v{C}}irka.
\newblock Rational approximations of holomorphic functions with singularities
  of finite order.
\newblock {\em Mat. Sb. (N.S.)}, 100(142)(1):137--155, 166, 1976.

\bibitem{Dem:1985}
Jean-Pierre Demailly.
\newblock Mesures de {M}onge-{A}mp\`ere et caract\'erisation g\'eom\'etrique
  des vari\'et\'es alg\'ebriques affines.
\newblock {\em M\'em. Soc. Math. France (N.S.)}, (19):124, 1985.

\bibitem{Dem:2012}
J.P. Demailly.
\newblock Complex analytic and differential geometry, (version of thursday june
  21, 2012).
\newblock Free accessible book
  (https://www-fourier.ujf-grenoble.fr/~demailly/documents.html), retrieved
  17.11.2015.

\bibitem{Gon:1972}
A.~A. Gon{\v{c}}ar.
\newblock Local conditions for the single-valuedness of analytic functions.
\newblock {\em Mat. Sb. (N.S.)}, 89(131):148--164, 167, 1972.

\bibitem{Gon:1974}
A.~A. Gon{\v{c}}ar.
\newblock A local condition for the single-valuedness of analytic functions of
  several variables.
\newblock {\em Mat. Sb. (N.S.)}, 93(135):296--313, 327, 1974.

\bibitem{Gon:1975}
A.~A. Gon{\v{c}}ar.
\newblock On a theorem of {S}aff.
\newblock {\em Mat. Sb. (N.S.)}, 94(136):152--157, 160, 1975.

\bibitem{Kli:1991}
Maciej Klimek.
\newblock {\em Pluripotential theory}, volume~6 of {\em London Mathematical
  Society Monographs. New Series}.
\newblock The Clarendon Press, Oxford University Press, New York, 1991.
\newblock Oxford Science Publications.

\bibitem{Rud:1968}
Walter Rudin.
\newblock A geometric criterion for algebraic varieties.
\newblock {\em J. Math. Mech.}, 17:671--683, 1967/1968.

\bibitem{Sic:1962}
J{\'o}zef Siciak.
\newblock On some extremal functions and their applications in the theory of
  analytic functions of several complex variables.
\newblock {\em Trans. Amer. Math. Soc.}, 105:322--357, 1962.

\bibitem{Sic:1981}
J{\'o}zef Siciak.
\newblock Extremal plurisubharmonic functions in {${\bf C}^{n}$}.
\newblock {\em Ann. Polon. Math.}, 39:175--211, 1981.

\bibitem{Skoda:1977}
Henri Skoda.
\newblock Morphismes surjectifs et fibr\'es lin\'eaires semi-positifs.
\newblock In {\em S\'eminaire {P}ierre {L}elong-{H}enri {S}koda ({A}nalyse),
  {A}nn\'ee 1976/77}, volume 694 of {\em Lecture Notes in Math.}, pages
  290--324. Springer, Berlin, 1978.

\bibitem{Stoll:1977}
Wilhelm Stoll.
\newblock {\em Value distribution on parabolic spaces}.
\newblock Lecture Notes in Mathematics, Vol. 600. Springer-Verlag, Berlin-New
  York, 1977.

\bibitem{Win:1970}
T.~Winiarski.
\newblock Approximation and interpolation of entire functions.
\newblock {\em Ann. Polon. Math.}, 23:259--273. (errata insert), 1970/1971.

\bibitem{Zak:1976}
V.~P. Zakharyuta.
\newblock Extremal plurisubharmonic functions, orthogonal polynomials and
  {B}ernstein-{W}alsh theorem for analytic functions of several complex
  variables.
\newblock {\em Ann. Polon. Math.}, 33:137--148, 1974 (Russian).

\bibitem{Zer:1991}
A.~Zeriahi.
\newblock Fonction de {G}reen pluricomplexe \`a p\^ole \`a l'infini sur un
  espace de {S}tein parabolique et applications.
\newblock {\em Math. Scand.}, 69(1):89--126, 1991.

\bibitem{Zer:1996}
A.~Zeriahi.
\newblock Approximation polynomiale et extension holomorphe avec croissance sur
  une vari\'et\'e alg\'ebrique.
\newblock {\em Ann. Polon. Math.}, 63(1):35--50, 1996.

\bibitem{Zer:2000}
Ahmed Zeriahi.
\newblock A criterion of algebraicity for {L}elong classes and analytic sets.
\newblock {\em Acta Math.}, 184(1):113--143, 2000.

\end{thebibliography}
\end{document}